\documentclass{amsart}
\usepackage{amsmath,amsfonts,amssymb,amsthm}
\usepackage{url}
\usepackage{graphicx,color}
\usepackage{enumerate}
\usepackage{amsthm}
\usepackage{amsmath}

\newtheorem{theorem}{Theorem}[section]
\newtheorem{lemma}[theorem]{Lemma}

\newtheorem*{theorem*}{Theorem}
\newtheorem{corollary}[theorem]{Corollary}
\theoremstyle{definition}
\newtheorem{definition}{Definition}[section]

\theoremstyle{remark}
\newtheorem*{remark}{Remark}
\numberwithin{equation}{section}

\author{}
\address{}


\begin{document}

\title{CAT(0) triangle-pentagon complexes}

\author[Ioana-Claudia Laz\u{a}r]{
Ioana-Claudia Laz\u{a}r\\
Politehnica University of Timi\c{s}oara, Dept. of Mathematics,\\
Victoriei Square $2$, $300006$-Timi\c{s}oara, Romania\\
E-mail address: ioana.lazar@upt.ro}

\begin{abstract}
We show that a certain triangulation of CAT(0) triangle-pentagon complexes is $7$-located and locally $5$-large.
Hereby we give examples of $7$-located, locally $5$-large groups. 

\end{abstract}

\maketitle

\section{Introduction}

Curvature can be expressed both in metric and combinatorial terms. On the metric side, one can refer to nonpositively  (negatively curved) in the sense of Aleksandrov and Gromov,
i.e. by comparing small triangles in the space with triangles in the Euclidean plane (hyperbolic space). Such triangles must satisfy the CAT(0) (CAT(-1))  inequality. In   \cite{Gr}, \cite{BH}, \cite{Ch}, \cite{Ch08}, \cite{crowley_2008}, \cite{AB}, \cite{AF}, \cite{GG} such metric curvature conditions are investigated. On the combinatorial side, one can express curvature 
using a condition, called local $k$-largeness ($k \geq 4$) which was introduced independently by Chepoi \cite{Ch} (under the name of bridged complexes), Januszkiewicz-{\' S}wi{\c a}tkowski \cite{JS1} and Haglund \cite{Hag}.
In \cite{BCCGO} a common generalization of systolic
and of CAT(0) cubical complexes is given. 
In \cite{CCGHO}, \cite{CCHO}, \cite{O-sdn}, \cite{ChOs}, \cite{O-8loc}, \cite{H}, \cite{HL}, \cite{L-8loc}, \cite{L-8loc2}, \cite{CCG} other combinatorial curvature conditions are studied.

In \cite{O-8loc} a local combinatorial condition called $8$-location is introduced and used to provide a new solution to Thurston's
problem about hyperbolicity of some $3$-manifolds. In \cite{L-8loc} we study a version of $8$-location, suggested in \cite[Subsection 5.1]{O-8loc}.
This $8$-location says that homotopically trivial loops of length at most $8$ admit filling diagrams with one internal vertex.
However, in the new $8$-location essential $4$-loops are allowed.
In \cite{L-8loc} (Theorem $4.3$) it is shown that simply connected,
$8$-located simplicial complexes are Gromov hyperbolic. 
In \cite{L-8loc2} another combinatorial curvature condition, called the $5/9$-condition, is introduced. It turns out that the complexes which satisfy
it, are $8$-located and therefore Gromov hyperbolic as well (see \cite{L-8loc2}). Besides, locally weakly systolic complexes (see \cite{O-sdn}) are $7$-located, locally $5$-large (see \cite{HL}; $7$-located complexes are considered in \cite{HL} as introduced initially in \cite{O-8loc}). Weakly systolic complexes appear in familiar environment since thickenings of CAT($-1$) cubical complexes are weakly systolic, but not
systolic (see \cite{O-sdn}). These arguments make $k$-located complexes ($k \geq 7$) worth being further investigated.

In the current paper we study the interplay between a certain metric curvature condition (given by the CAT(0) metric) and a certain combinatorial curvature condition (given by $7$-location). Such relation is studied on a particular triangulation of a triangle-pentagon complex presented below. A similar connection between the two curvature conditions was investigated before (see \cite{HL}, Theorem $4.6$). Namely, in \cite{HL} it is shown that a certain metric on a $7$-located simplicial disc is locally CAT(0). However, the definition of $7$-location referred to in \cite{HL} (which is the one considered in \cite{O-8loc}) differs from the one we shall consider. More precisely, we refer to the definition introduced in \cite{O-8loc} and applied further in \cite{L-8loc}. 

Next we present the main construction of the paper. We consider a CAT(0) triangle-pentagon complex $X$ and subdivide it as follows. We subdivide each pentagon of $X$ into five triangles by considering a new vertex, namely the center of each pentagon and joining it with the vertices of the corresponding pentagon. Hence we obtain a simplicial complex $X^{\star}$ whose vertex set is the union of the set of the vertices of $X$ and the centers of the pentagons of $X$. The edges of $X^{\star}$ are given by the union of the edges of $X$ and the edges joining the centers of the pentagons with the vertices of the corresponding pentagons. The $2$-simplices of $X^{\star}$ are given by the union of the triangles of $X$ and the triangles obtained by subdividing each pentagon of $X$ into five triangles. Note that each pentagon of $X$ is replaced in $X^{\star}$ by $5$ triangles.

The main purpose of the paper is to show.

%Mention construction by Chepoi-Chalopin according to %which $X^{2}$ is dismantable. Consider a triangulation %of the triangle-pentagon complex.

\begin{theorem*}  (Lemma $3.2$, Theorem $3.3$) 
A CAT(0) triangle-pentagon complex admits a triangulation which, when endowed with a particular metric, is $7$-located and locally $5$-large.
\end{theorem*}

The result above resembles in some sense the one shown on triangle-pentagon complexes of CB-graphs (\cite{CCG}). Namely, in order to obtain dismantability, triangle-pentagon complexes of CB-graphs are triangulated in a certain way as well (see \cite{CCG}, chapter $5.1$).

As corollary, we obtain examples of $7$-located, locally $5$-large groups.

\subsection{Structure of the paper} In Section $2$ we present basic definitions,
notation and results; in Section $3$ we prove Theorem $3.3$.

\subsection{Acknowledgements:} I thank Giovanni Sartori for a useful remark.

\section{Preliminaries}

\subsection{Generalities}

Let $X$ be a simplicial complex.
We denote by $X^{(k)}$ the $k$-skeleton of $X, 0 \leq k < \dim X$.
A subcomplex $L$ in $X$ is called \emph{full} as a subcomplex of $X$ if any simplex of $X$ spanned by a set of vertices in $L$, is a simplex of $L$.
For a set $A = \{ v_{1}, ..., v_{k} \}$ of vertices of $X$, by $\langle A \rangle$ or by $\langle  v_{1}, ..., v_{k} \rangle$ we denote the \emph{span} of $A$, i.e. the smallest full subcomplex of $X$ that contains $A$.
We write $v \sim v'$ if $\langle  v,v' \rangle \in X$ (it can happen that $v = v'$).
We write $v \nsim v'$ if $\langle  v,v' \rangle \notin X$.
We call $X$ {\it flag} if any finite set of vertices which are pairwise connected by edges of $X$, spans a simplex of $X$.

We define the \emph{combinatorial metric} on the $0$-skeleton of $X$ as the number of edges in the shortest $1$-skeleton path joining two given vertices.

A {\it cycle} ({\it loop}) $\gamma$ in $X$ is a subcomplex of $X$ isomorphic to a triangulation of $S^{1}$.
A \emph{full cycle} in $X$ is a cycle that is full as a subcomplex of $X$. 
A \emph{$k$-wheel} in $X$ $(v_{0}; v_{1}, ..., v_{k})$ (where $v_{i}, i \in \{0, ..., k\}$ are vertices of $X$) is a subcomplex of $X$ such that $\gamma = (v_{1}, ..., v_{k})$ is a full cycle and $v_{0} \sim v_{1}, ..., v_{k}$. The \emph{length} of $\gamma$ (denoted by $|\gamma|$) is the number of edges of $\gamma$.

\subsection{Triangle-pentagon complexes}

Triangle-pentagon complexes are introduced and studied in \cite{CCG}.

\begin{definition}
We define a \emph{triangle-pentagon complex} of a graph $G$ as
a two-dimensional cell complex with $1$-skeleton $G$, and such that the two-cells are (solid) triangles and
pentagons whose boundaries are identified by isomorphisms with (graph) triangles and pentagons in $G$. There are no full $4$-cycles in a triangle-pentagon complex. Pentagon of a triangle-pentagon complex have no diagonals. 
\end{definition}

Below we introduce the main construction of the paper.

\begin{definition}
Let $X$ be a triangle-pentagon complex. Let $X_{p}$ denote the set of centers of the pentagons of $X$. We define $X^{\star}$ to be a $2$-dimensional simplicial complex  whose vertex set is the union of all vertices of $X$ and the set of vertices $X_{p}$. The edges of $X^{\star}$ is the union of the edges of $X$ and a set of edges joining the vertices of $X_{p}$ to the vertices of the corresponding pentagons. The triangles of $X^{\star}$ are the union of the triangles of $X$ and the triangles obtained by dividing each pentagon into $5$ triangles by considering the centers of the pentagons. We call the pentagons triangulated in this way, \emph{centered pentagons} or  \emph{$5$-wheels}. Note that in the simplicial complex $X^{\star}$ each pentagon of the triangle-pentagon complex $X$, is replaced by five triangles. 
\end{definition}

 \subsection{CAT(0) spaces}

\begin{definition}
Let $(X,d)$ be a geodesic space. Let $\triangle (p,q,r)$ be a geodesic
triangle in $X$. Let $\overline{\triangle}
(\overline{p},\overline{q},\overline{r}) \subset \mathbf{R}^{2}$ be
a comparison triangle for $\triangle$. The metric $d$ is
\emph{CAT(0)} if for all $x,y \in \triangle$ and all comparison
points $\overline{x}, \overline{y} \in \overline{\triangle}$, the
CAT(0) inequality holds 
$d(x,y) \leq d_{\mathbf{R}^{2}}(\overline{x}, \overline{y})$.
\end{definition}

\begin{definition}\label{1.2.2}
A geodesic space $X$ is called a \emph{CAT(0) space} if it is a geodesic
space all of whose geodesic triangles satisfy the CAT(0) inequality.
\end{definition}

%Define $\rm{Shapes}(X)$.

\begin{theorem} Let $X$ be a $2$-dimensional simplicial complex with $\rm{Shapes}(X)$ finite. Then $X$ is locally a CAT(0) space if and only
if for each vertex $v \in X$, each injective loop in $X_{v}$ has length at least $2 \pi$.
(see \cite{BH}, Theorem $5.5$, Lemma $5.6$)
\end{theorem}

\subsection{Local $k$-largeness}

Let $\sigma$ be a simplex of $X$.
The \emph{link} of $X$ at $\sigma$, denoted $X_{\sigma}$, is the subcomplex of $X$ consisting of all simplices of $X$ which are disjoint from $\sigma$ and which, together with $\sigma$, span a simplex of $X$.
We call a flag simplicial complex $k$\emph{-large}, $k \geq 4$, if there are no full $j$-cycles in $X$, for $j<k$.
We say $X$ is \emph{locally} $k$\emph{-large}, $k \geq 4$, if all its links are $k$-large.  We call a vertex of $X$ $k$-large if its link is $k$-large.

\subsection{$7$-location}

We introduce further a global combinatorial condition on a flag simplicial complex.

%\begin{definition}\label{def-2.2}
%A simplicial complex is $m$-\emph{located}, $m \geq %4$, if it is flag and every full homotopically trivial %loop of length at most $m$ is contained in a $1$-ball.
%\end{definition}

\begin{definition}
A $(k,l)$-dwheel $W = W_{1} \cup W_{2} = (w_{l}; v_{1},v_{2}, v_{3} = w_{l-1},\cdots,$ $ v_{k}) \cup (v_{2}; w_{1}, w_{2}, \cdots, w_{l})$ is the union of two full wheels with $v_{1}=w_{1}$ or $v_{1} \sim w_{1}$. If $v_{1}=w_{1}$, we call $W$ \textbf{planar dwheel}. If $v_{1} \sim w_{1}$, we call $W$ \textbf{nonplanar dwheel}. The boundary length of $W$ is $k+l-4$ if $W$ is planar and $k+l-3$ if it is nonplanar.
\end{definition}

        \begin{figure}[h]
            \begin{center}
               \includegraphics[height=4cm]{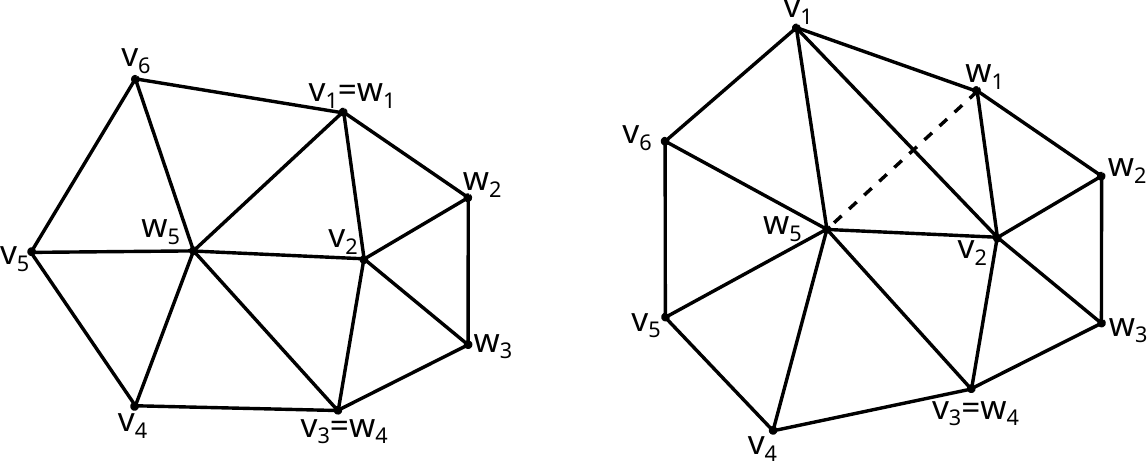}
               % \label{$w_{1} \sim w_{3}$}
              \caption{Planar dwheel and nonplanar dwheel}
            \end{center}
        \end{figure}

\begin{definition}\label{2.4}
A simplicial complex is $m$-\emph{located}, $m \geq 4$ if it is flag and whenever a dwheel subcomplex $W = W_{1} \cup W_{2}$ of $X$ satisfies the following conditions:
\begin{enumerate}
\item $\partial W$ has length at most $m$;
\item the wheels $W_{1}$ and $W_{2}$ are full subcomplexes of $X$,
\end{enumerate} the dwheel subcomplex $W$ is contained in the link $X_{v}$ of some vertex $v$.
\end{definition}

\begin{definition}
Let $X$ be a simplicial complex and let $v$ be a vertex of $X$. We call \emph{combinatorial girth} of $v$ the number of edges in $X_{v}$.
\end{definition}

\begin{remark}
Because triangle-pentagon complexes are $2$-dimensional, they do not contain nonplanar dwheels.
\end{remark}

\begin{definition}\label{2.2}
We say that a flag simplicial complex satisfies the \emph{$5/8$-condition}, or that it is a \emph{$5/8$-complex}, if it satisfies the following two conditions:
\begin{description}
\item[($5/8$)]
every vertex adjacent to a $4$-large (but not $5$-large) vertex is $8$-large;
\item[($6/7$)]
every vertex adjacent to a $5$-large (but not $6$-large) vertex is $7$-large.
\end{description}
\end{definition}

\subsection{Geometric group actions}

A group $\Gamma$ \textit{acts} by automorphisms on a triangle-pentagon complex $X$ if there is a homomorphism $\Gamma \rightarrow \rm{Aut}(X)$ called
an \textit{action} of $\Gamma$. The action is \emph{geometric} (or $\Gamma$ acts geometrically) if it is proper (i.e., cells stabilizers are
finite) and cocompact (i.e., the quotient $X|_{\Gamma}$ is compact). In the current paper we usually consider
geometric actions on graphs, viewed as one-dimensional complexes. Namely, we say that a group $\Gamma$ acts
on a graph $G$ when it acts on $G$ if we consider $G$ as a $1$-dimensional simplicial complex. Observe that if $\Gamma$ acts on $G$, then it induces a group action of $\Gamma$ on its clique complex $X(G)$. Note that this also induces an action of $\Gamma$ on the triangle-pentagon complex of $G$.

\section{The main result}

We show that a certain subdivision of a CAT(0)  triangle-pentagon complex is $7$-located and locally $5$-large.

\begin{definition}
Let $W_{5}$ be a $5$-wheel with central vertex $v$. We metrize each triangle of $W$ as an Euclidean triangle whose angle at $v$ is $\cfrac{2 \pi}{5}$, whose remaining two angles are both equal to $\cfrac{3 \pi}{10}$ and whose boundary edge $e \subset \partial W$ has length $1$. We then metrize $W$ as an Euclidean polygonal complex. We call the resulting metric the \emph{flattened wheel metric} on $W$.
\end{definition}

\begin{remark}
A $5$-wheel with the flattened wheel metric is isometric to a regular Euclidean pentagon of side length $1$. We call such $5$-wheel a \emph{flattened $5$-wheel}. Under this isometry, the central vertex of the $5$-wheel is sent to the center of the pentagon.
\end{remark}

\begin{definition}
Let $X$ be a triangle-pentagon complex endowed with a metric $d$. We endow $X^{\star}$ with a metric $d^{\star}$ as follows:
\begin{itemize}

\item at any vertex $v \in X^{(0)}$, $d = d^{\star}$, i.e., for any interior vertex $v$ of $X$, the sum of the measures of the angles around it in $X$ equals the sum of the measures of the angles around it in $X^{\star}$;

\item on any $W_{5}$, let $d^{\star}$ be the flattened wheel metric; then $W_{5}$ becomes a flattened $5$-wheel;

\item let any triangle of $X^{\star}$ which is also a triangle of $X$, be equilateral and therefore let any of its angles be of measure $\cfrac{\pi}{3}$. 
\end{itemize}
\end{definition}

\begin{lemma}\label{3.1}
Let $X$ be a triangle-pentagon complex endowed with a CAT(0) metric $d$. Let $X^{\star}$ be endowed with the metric $d^{\star}$. Then $d^{\star}$ is a CAT(0) metric.
\end{lemma}

\begin{proof}
Let $v$ be the center of a $5$-wheel of $X$, i.e., $v$ is a vertex of $X^{\star}$.
Because $d^{\star}$ is the flattened wheel metric with which $W_{5}$ is endowed, the sum of the angles around $v$ is equal to $2 \pi$.  Due to the construction of $X^{\star}$, the metric $d^{\star}$ is well defined.

Because $d$ is a CAT(0) metric and $X$ has a finite number of isometry types of $2$-cells, the sum of the measures of the angles around each vertex of $X$ is at least $2 \pi$.

Note that $X^{\star}$ has a single isometry type of $2$-cells. Due to the definition of $d^{\star}$, the sum of the measures of the angles around each vertex of $X^{\star}$, is at least $2 \pi$. In conclusion, Theorem \ref{2.1} imlpies that $X^{\star}$ is locally a CAT(0) space. Because $X$ is a simply connected $2$-complex while $X^{\star}$ is a subdivision of $X$, $X^{\star}$ is also simply connected. Hence $X^{\star}$ is a CAT(0) space.
\end{proof}

\begin{lemma}\label{3.2}
Let $X$ be a CAT(0) triangle-pentagon complex.  Let $X^{\star}$ be endowed with the metric $d^{\star}$.  Then there are no $4$-large and no $3$-large vertices in $X^{\star}$. In particular, $X^{\star}$ is flag and locally $5$-large. Moreover, any $5$-large vertex of $X^{\star}$ is the center of a pentagon of $X$.
\end{lemma}

\begin{proof}
Due to Lemma \ref{3.1}, $X^{\star}$ endowed with the metric $d^{\star}$, is a CAT(0) space.

The proof is by contradiction.
Suppose there exists in $X^{\star}$ a $3$-large vertex $v$. There are two cases treated below.

 $\bullet$ Let $v$ be the vertex of a pentagon $\tau = \langle v = v_{1}, v_{2}, v_{3}, v_{4}, v_{5} \rangle$ of $X$. Let $w \in {X^{\star}}^{(0)}$ be the center of $\tau$. Because $v$ is $3$-large, $X_{v} = (v_{5}, v_{2}, w)$. Then the sum of the measures of the angles around $v$ is
 \begin{center}
$2 \cdot \cfrac{3 \pi}{10} + \cfrac{\pi}{3} = \cfrac{14 \pi}{15} < 2 \pi.$
 \end{center} This yields a contradiction with the fact that $X^{\star}$ is a CAT(0) space. In fact, because $\tau$ is a pentagon, we have $v_{2} \nsim v_{5}$. Therefore such case can in fact not occur.
 
 $\bullet$ Let $v$ be the common vertex of three equilateral triangles $\langle v,w_{i},w_{i+1} \rangle, 1 \leq i \leq 2$, $\langle v,w_{1},w_{3} \rangle$. Because $v$ is $3$-large, $X_{v} = (w_{1}, w_{2}, w_{3})$. Then the sum of the measures of the angles around $v$ is
 \begin{center}
$3 \cdot \cfrac{\pi}{3} = \pi < 2 \pi.$
 \end{center}This yields a contradiction with the fact that $X^{\star}$ is a CAT(0) space. 

 In conclusion $X^{\star}$ contains no $3$-large vertices. Therefore $X^{\star}$ is flag.

Suppose in $X^{\star}$ there exists a $4$-large vertex $v$. There are two cases treated below.

$\bullet$ Let $v$ be the vertex of a pentagon  $\tau = \langle v = v_{1}, v_{2}, v_{3}, v_{4}, v_{5} \rangle$ of $X$. Let $w_{2} \in {X^{\star}}^{(0)}$ be the center of $\tau$. Then $X_{v} = (w_{1}=v_{5}, w_{2}, w_{3}=v_{2}, w_{4})$. Besides $v$ is a common vertex of the equilateral triangle $\langle w_{1},v,w_{4} \rangle$ and $\langle w_{4},v,w_{3} \rangle$. Hence the sum of the measures of the angles around $v$ is
\begin{center}
$2 \cdot \cfrac{3 \pi}{10} + 2 \cdot \cfrac{\pi}{3} = \cfrac{19 \pi}{15} < 2 \pi.$
\end{center} This yields a contradiction with the fact that $X^{\star}$ is a CAT(0) space. 

$\bullet$ Let $v$ be the common vertex of the equilateral triangles $\langle v, w_{i}, w_{i+1} \rangle, 1 \leq i \leq 3$, $\langle v, w_{1}, w_{4} \rangle$. Then $X_{v} = (w_{1}, w_{2}, w_{3}, w_{4})$. Hence the sum of the measures of the angles around $v$ is
\begin{center}
$4 \cdot \cfrac{\pi}{3} < 2 \pi.$
\end{center} This yields a contradiction with the fact that $X^{\star}$ is a CAT(0) space. 

In conclusion, because $X^{\star}$ contains no $3$-large and no $4$-large vertices, it is locally $5$-large.

\end{proof}

\begin{theorem}\label{3.3}
Let $X$ be a CAT(0) triangle-pentagon complex.  Let $X^{\star}$ be endowed with the metric $d^{\star}$. Then $X^{\star}$ is $7$-located. 
\end{theorem}

\begin{proof}
Lemma \ref{3.1} implies that $d^{\star}$ is a CAT(0) metric.

We show that any loop $\gamma$ in $X^{\star}$ of length at most $7$ at the boundary of a dwheel, is in the link of a vertex, i.e., there is a vertex $u$ of $X^{\star}$ such that $\gamma \subset X^{\star}_{u}$. 
Due to the Lemma \ref{3.2}, in $X^{\star}$ there are no $3$-large and no $4$-large vertices.

We show that the only $5$-large vertices of $X^{\star}$ are the centers of the pentagons of $X$. Suppose in $X^{\star}$ there exists a $5$-large vertex $v$ which is not the center of a pentagon of $X$. Then there are two cases treated below.

$\bullet$ Suppose $v$ lies at the intersection of a pentagon of $X$ and three equilateral triangles of $X$ (and therefore also of $X^{\star}$). Hence in $X^{\star}$ the sum of the measures of the angles at $v$ is
\begin{center}
$2 \cdot \cfrac{3 \pi}{10} + 3 \cdot \cfrac{\pi}{3} < 2 \pi.$
\end{center} This yields a contradiction with $X^{\star}$ being a CAT(0) space.

$\bullet$ Suppose $v$ lies at the intersection of five equilateral triangles of $X$ (and therefore also of $X^{\star}$). Hence in $X^{\star}$ the sum of the measures of the angles at $v$ is
\begin{center}
$5 \cdot \cfrac{\pi}{3} < 2 \pi.$
\end{center} This yields a contradiction with $X^{\star}$ being a CAT(0) space. 

In conclusion the only $5$-large vertices $v$ of $X^{\star}$ are the centers of pentagons. Let $v$ be the center of the pentagon $\tau = \langle v_{1}, v_{2}, v_{3}, v_{4}, v_{5} \rangle$ of $X$. Then $X_{v} = (v_{1}, v_{2}, v_{3}, v_{4}, v_{5})$.
Let $w=v_{2}$ be a vertex of $X$ adjacent to $v$. Due to Lemma \ref{3.2}, $w$ is not $3$-large and it is not $4$-large.

$\bullet$ Suppose $w$ is $5$-large. Let $X_{w} = (w_{1}=v_{1}, w_{2}, w_{3}, w_{4}=v_{3}, w_{5} = v)$. Then $\gamma = (w_{1}=v_{1}, w_{2}, w_{3}, w_{4}=v_{3}, v_{4}, v_{5})$ is a loop in $X^{\star}$ at the boundary of a dwheel. Note that $|\gamma| = 6$. Note that the sum of the measures of the angles around $w$ is
\begin{center}
$2 \cdot \cfrac{3 \pi}{10} + 3 \cdot \cfrac{\pi}{3} = \cfrac{8 \pi}{5} < 2 \pi$. \end{center} This yields a contradiction with the fact that $X^{\star}$ is a CAT(0) space.

$\bullet$ Suppose $w$ is $6$-large. Let $X_{w} = (w_{1}=v_{1}, w_{2}, w_{3}, w_{4}, w_{5}=v_{3}, w_{6} = v)$. Then $\gamma =(w_{1}=v_{1}, w_{2}, w_{3}, w_{4}, w_{5} = v_{3}, v_{4}, v_{5})$ is a loop in $X^{\star}$ at the boundary of a dwheel. Note that $|\gamma| = 7$. Hence the sum of the measures of the angles around $w$ is
\begin{center}
$2 \cdot \cfrac{3 \pi}{10} + 4 \cdot \cfrac{\pi}{3} = \cfrac{29 \pi}{15} < 2 \pi$.\end{center} This yields a contradiction with the fact that $X^{\star}$ is a CAT(0) space.

In conclusion in $X^{\star}$ there are no loops $\gamma$ of length at most $7$ at the boundary of a dwheel. Any loop in $X^{\star}$ at the boundary of a dwheel and of length at most $7$, is therefore filled with a single vertex. Due to Lemma \ref{3.2}, $X^{\star}$ is flag. In conclusion $X^{\star}$ is $7$-located.

%!! Still need to argue why it can not be filled with 3 %vertices. Analogue 5/9-condition. (if we use the other %definition; not the one with dwheels)

\end{proof}

The theorem shown above enables to give examples of $7$-located, locally $5$-large groups.

\begin{corollary}\label{3.4}
    Let $G$ be a group acting geometrically on the $1$-skeleton of a CAT(0) triangle-pentagon complex $X$. Then $G$ is $7$-located and locally $5$-large.
\end{corollary}

\begin{proof}

The geometric action of $G$ on $X^{(1)}$ induces an action of $G$ on its clique complex $X$. Hence $G$ induces a geometric action on $X^{\star}$, the subdivision of $X$. Due to Theorem \ref{3.3}, $X^{\star}$ is $7$-located and locally $5$-large. Therefore, so is $G$.

\end{proof}

\begin{corollary}\label{3.5}
Let $X$ be a triangle-pentagon complex endowed with a CAT(0) metric. Then $X^{\star}$ is a $5/8$-complex (only the second condition of the definition is fulfilled).
\end{corollary}

\begin{proof}
The Corollary follows by Theorem \ref{3.3}. Because $X^{\star}$ has no $4$-large vertices, the first condition in the definition of $5/8$-complexes, is empty.
\end{proof}

Let $v_{1} \sim v_{2} \sim v_{3}$ be vertices of a complex $X$.
We denote by $m(\widehat{v_{1}v_{2}v_{3}})$ the measure of the angle $\widehat{v_{1}v_{2}v_{3}}$.

\begin{theorem}\label{3.6}
Let $X$ be a triangle-pentagon complex endowed with a CAT(0) metric. Let $X^{\star}$ be endowed with the metric $d^{\star}$. Let $v$ be a vertex of a pentagon of $X$. Then the combinatorial girth in $X^{\star}$ of $v$, is at least $7$.
\end{theorem}

\begin{proof}
Because $X$ is a CAT(0) space, Lemma \ref{3.1} implies that $d^{\star}$ is a CAT(0) metric.

We distinguish four cases which are treated below.

$\bullet$ Let $v$ be an interior vertex of $X$ which is a common vertex of four pentagons $\tau_{i}, 1 \leq i \leq 4$ of $X$ pairwise intersecting along their common edges. Let $x_{i}$ be the center of the pentagon $\tau_{i}, 1 \leq i \leq 4$. Let $\tau_{1} = \langle v,v_{1}, v_{2}, v_{3}, v_{4} \rangle$, $\tau_{2} = \langle v, v_{4}, v_{5}, v_{6}, v_{7} \rangle$, $\tau_{3} = \langle v,v_{7}, v_{8}, v_{9}, v_{10} \rangle$, $\tau_{4} = \langle v,v_{10}, v_{11}, v_{12}, v_{1} \rangle$. Note that \begin{center}
$m(\widehat {v_{1} v x_{1}}) + m(\widehat {x_{1} v v_{4}}) + m(\widehat {v_{4} v x_{2}}) + m(\widehat {x_{2} v v_{7}}) +$ 
\end{center}
\begin{center}
$ + m(\widehat {v_{7} v x_{3}}) + m(\widehat {x_{3} v v_{10}}) + m(\widehat {v_{10} v x_{4}}) + m(\widehat {x_{4} v v_{1}}) = 8 \cdot \cfrac{3 \pi}{10} = \cfrac{12 \pi}{5} > 2 \pi$.\end{center} Hence such configuration does not spoil the CAT(0) metric, $X^{\star}$ is endowed with.
 Note that $|X^{\star}_{v}| = 8$. Then the combinatorial girth of $v$ in $X^{\star}$ is $8$.

\begin{figure}[h]
            \begin{center}
               \includegraphics[height=6cm]{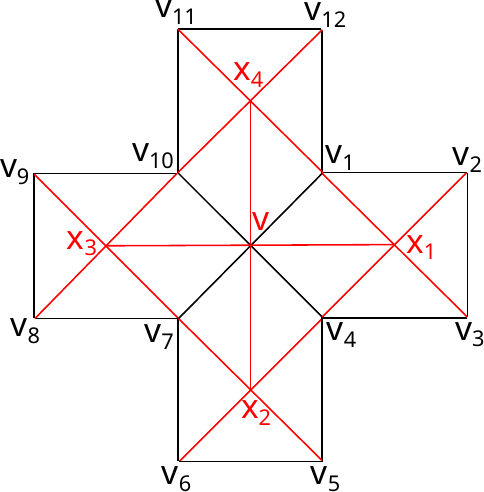}
               % \label{$w_{1} \sim w_{3}$}
              \caption{}
            \end{center}
        \end{figure}

$\bullet$ Let $v$ be an interior vertex of $X$ which is a common vertex of three pentagons $\tau_{i}, 1 \leq i \leq 3$ of $X$ such that $\tau_{1}$, $\tau_{2}$, and $\tau_{2}$, $\tau_{3}$, respectively, pairwise intersect along their common edges. Let $x_{i}$ be the center of the pentagon $\tau_{i}, 1 \leq i \leq 3$. Let $\tau_{1} = \langle v,v_{1}, v_{2}, v_{3}, v_{4} \rangle$, $\tau_{2} = \langle v, v_{4}, v_{5}, v_{6}, v_{7} \rangle$, $\tau_{3} = \langle v,v_{7}, v_{8}, v_{9}, v_{10} \rangle$. We discuss the cases $v_{1} = v_{10}$ and $v_{1} \sim v_{10}$.

\begin{enumerate}
\item If $v_{1} = v_{10}$, we have
\begin{center}
$m(\widehat {v_{1} v x_{1}}) + m(\widehat {x_{1} v v_{4}}) + m(\widehat {v_{4} v x_{2}}) + m(\widehat {x_{2} v v_{7}}) +$ 
\end{center}
\begin{center}
$ + m(\widehat {v_{7} v x_{3}}) + m(\widehat {x_{3} v v_{10}}) = 6 \cdot \cfrac{3 \pi}{10} = \cfrac{9 \pi}{5} < 2 \pi$
\end{center}  which yields a contradiction with $X^{\star}$ being a CAT(0) space. Therefore $v_{1} \neq v_{10}$.

\item If $v_{1} \sim v_{10}$, we have
\begin{center}
$m(\widehat {v_{1} v x_{1}}) + m(\widehat {x_{1} v v_{4}}) + m(\widehat {v_{4} v x_{2}}) + m(\widehat {x_{2} v v_{7}}) +$ 
\end{center}
\begin{center}
$ + m(\widehat {v_{7} v x_{3}}) + m(\widehat {x_{3} v v_{10}}) + m(\widehat {v_{10} v v_{1}}) = 6 \cdot \cfrac{3 \pi}{10} + \cfrac{\pi}{3} = \cfrac{32 \pi}{15} > 2 \pi$.
\end{center} Hence such configuration does not spoil the CAT(0) metric, $X^{\star}$ is endowed with. Note that $|X^{\star}_{v}| = 7$. Then the combinatorial girth of $v$ in $X^{\star}$ is $7$.

\begin{figure}[h]
            \begin{center}
               \includegraphics[height=5cm]{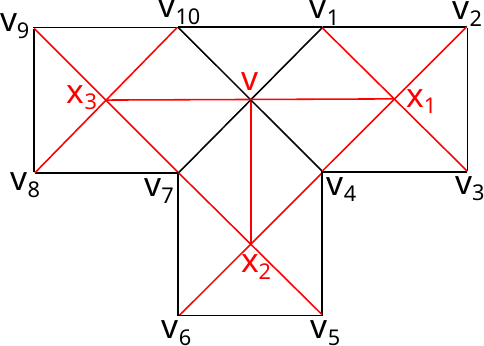}
               % \label{$w_{1} \sim w_{3}$}
              \caption{}
            \end{center}
        \end{figure}
\end{enumerate}

$\bullet$ Let $v$ be an interior vertex of $X$ which is the common vertex of two pentagons $\tau_{i}, 1 \leq i \leq 2$ of $X$ intersecting along their common edge.
Let $x_{i}$ be the center of the pentagon $\tau_{i}, 1 \leq i \leq 2$. Let $\tau_{1} = \langle v,v_{1}, v_{2}, v_{3}, v_{4} \rangle$, $\tau_{2} = \langle v, v_{4}, v_{5}, v_{6}, v_{7} \rangle$.  We discuss the cases $v_{1} = v_{7}$ and $1 \leq d(v_{1},v_{7}) \leq 3$.  

\begin{enumerate}
\item If $v_{1} = v_{7}$, we have
\begin{center}
$m(\widehat {v_{1} v x_{1}}) + m(\widehat {x_{1} v v_{4}}) + m(\widehat {v_{4} v x_{2}}) + m(\widehat {x_{2} v v_{7}}) =$
\end{center}
\begin{center}
$= 4 \cdot \cfrac{3 \pi}{10}  = \cfrac{6 \pi}{5} < 2 \pi$
\end{center}  which yields a contradiction because $X^{\star}$ is a CAT(0) space. Therefore $v_{1} \neq v_{7}$.

\item If $v_{1} \sim v_{7}$, we have
\begin{center}
$m(\widehat {v_{1} v x_{1}}) + m(\widehat {x_{1} v v_{4}}) + m(\widehat {v_{4} v x_{2}}) + m(\widehat {x_{2} v v_{7}}) + m(\widehat {v_{7} v v_{1}}) =$
\end{center}
\begin{center}
$= 4 \cdot \cfrac{3 \pi}{10} + \cfrac{\pi}{3} = \cfrac{23 \pi}{15} < 2 \pi$
\end{center}  which yields a contradiction because $X^{\star}$ is a CAT(0) space. Therefore $d(v_{1},v_{7}) > 1$.

\item If $d(v_{1},v_{7}) = 2$, i.e., there exists a vertex $v_{8}$ such that $v_{1} \sim v_{8} \sim v_{7}$. Then we have
\begin{center}
$m(\widehat {v_{1} v x_{1}}) + m(\widehat {x_{1} v v_{4}}) + m(\widehat {v_{4} v x_{2}}) + m(\widehat {x_{2} v v_{7}}) +$
\end{center}
\begin{center}
$+ m(\widehat {v_{7} v v_{8}})  + m(\widehat {v_{8} v v_{1}})= 4 \cdot \cfrac{3 \pi}{10} + 2 \cdot \cfrac{\pi}{3} = \cfrac{28 \pi}{15} < 2 \pi$
\end{center} which yields a contradiction because $X^{\star}$ is a CAT(0) space. Therefore $d(v_{1},v_{7}) > 2$.

\item If $d(v_{1},v_{7}) = 3$, i.e., there exist vertices $v_{8}, v_{9}$ such that $v_{1} \sim v_{9} \sim v_{8} \sim v_{7}$. Then we have
\begin{center}
$m(\widehat {v_{1} v x_{1}}) + m(\widehat {x_{1} v v_{4}}) + m(\widehat {v_{4} v x_{2}}) + m(\widehat {x_{2} v v_{7}}) +$
\end{center}
\begin{center}
$+ m(\widehat {v_{7} v v_{8}})  + m(\widehat {v_{8} v v_{9}}) + m(\widehat {v_{9} v v_{1}})= 4 \cdot \cfrac{3 \pi}{10} + 3 \cdot \cfrac{\pi}{3} = \cfrac{11 \pi}{5} > 2 \pi$.
\end{center} Hence such configuration does not spoil the CAT(0) metric, $X^{\star}$ is endowed with. Note that $|X^{\star}_{v}| = 7$. Then the combinatorial girth of $v$ in $X^{\star}$ is $7$.

\begin{figure}[h]
            \begin{center}
               \includegraphics[height=5cm]{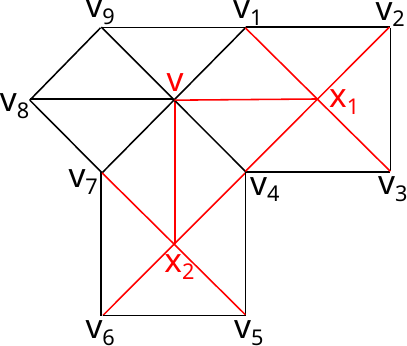}
               % \label{$w_{1} \sim w_{3}$}
              \caption{}
            \end{center}
        \end{figure}
\end{enumerate}

$\bullet$ Let $v$ be an interior vertex of $X$ which is a vertex of a pentagon  $\tau = \langle v,v_{1}, v_{2}, v_{3}, v_{4} \rangle$ of $X$.
Let $x_{1}$ be the center of the pentagon $\tau$.  We discuss the cases $1 \leq d(v_{1},v_{4}) \leq 5$. 

\begin{enumerate}
\item If $v_{1} \sim v_{4}$, we have
\begin{center}
$m(\widehat {v_{1} v x_{1}}) + m(\widehat {x_{1} v v_{4}}) + m(\widehat {v_{4} v v_{1}}) = 2 \cdot \cfrac{3 \pi}{10} + \cfrac{\pi}{3} = \cfrac{14 \pi}{15} < 2 \pi$
\end{center}  which yields a contradiction with $X^{\star}$ being a CAT(0) space. In fact, because $\tau_{1} = \langle v,v_{1}, v_{2}, v_{3}, v_{4} \rangle$ is a pentagon, we have $v_{1} \nsim v_{4}$. Hence such case can not even occur.

\item If $d(v_{1}, v_{4}) = 2$, i.e., there exists a vertex $v_{5}$ such that $v_{4} \sim v_{5} \sim v_{1}$. We have
\begin{center}
$m(\widehat {v_{1} v x_{1}}) + m(\widehat {x_{1} v v_{4}}) + m(\widehat {v_{4} v v_{5}}) + m(\widehat {v_{5} v v_{1}}) =$
\end{center}
\begin{center}
$= \cfrac{3 \pi}{5} + 2 \cdot \cfrac{\pi}{3} = \cfrac{19 \pi}{15} < 2 \pi$
\end{center}  which yields a contradiction with $X^{\star}$ being a CAT(0) space. Therefore $d(v_{1},v_{4}) > 2$.

\item If $d(v_{1}, v_{4}) = 3$, i.e., there exist vertices $v_{5}, v_{6}$ such that $v_{4} \sim v_{5} \sim v_{6} \sim v_{1}$. Then we have
\begin{center}
$m(\widehat {v_{1} v x_{1}}) + m(\widehat {x_{1} v v_{4}}) + m(\widehat {v_{4} v v_{5}}) + m(\widehat {v_{5} v v_{6}}) +$
\end{center}
\begin{center}
$+ m(\widehat {v_{6} v v_{1}}) = 2 \cdot \cfrac{3 \pi}{10} + 3 \cdot \cfrac{\pi}{3} = \cfrac{8 \pi}{5} < 2 \pi$
\end{center}  which yields a contradiction with $X^{\star}$ being a CAT(0) space. Therefore $d(v_{1},v_{4}) > 3$.

\item If $d(v_{1}, v_{4}) = 4$, i.e., there exist vertices $v_{5}, v_{6}, v_{7}$ such that $v_{4} \sim v_{5} \sim v_{6} \sim v_{7} \sim v_{1}$. Then we have
\begin{center}
$m(\widehat {v_{1} v x_{1}}) + m(\widehat {x_{1} v v_{4}}) + m(\widehat {v_{4} v v_{5}}) + m(\widehat {v_{5} v v_{6}}) +$
\end{center}
\begin{center}
$+ m(\widehat {v_{6} v v_{7}}) + m(\widehat {v_{7} v v_{1}}) = 2 \cdot \cfrac{3 \pi}{10} + 4 \cdot \cfrac{\pi}{3} = \cfrac{29 \pi}{15} < 2 \pi$
\end{center} which yields a contradiction with $X^{\star}$ being a CAT(0) space. Therefore $d(v_{1},v_{4}) > 4$.

\item If $d(v_{1}, v_{4}) = 5$, i.e., there exist vertices $v_{5}, v_{6}, v_{7}, v_{8}$ such that $v_{4} \sim v_{5} \sim v_{6} \sim v_{7} \sim v_{8} \sim v_{1}$. Then we have
\begin{center}
$m(\widehat {v_{1} v x_{1}}) + m(\widehat {x_{1} v v_{4}}) + m(\widehat {v_{4} v v_{5}}) + m(\widehat {v_{5} v v_{6}}) +$ 
\end{center}
\begin{center}
$+ m(\widehat {v_{6} v v_{7}}) + m(\widehat {v_{7} v v_{8}}) + m(\widehat {v_{8} v v_{1}}) =$
\end{center}
\begin{center}
$= 2 \cdot \cfrac{3 \pi}{10} + 5 \cdot \cfrac{\pi}{3} = \cfrac{34 \pi}{15} > 2 \pi$.
\end{center}  Hence such configuration does not spoil the CAT(0) metric, $X^{\star}$ is endowed with. Note that $|X^{\star}_{v}| = 7$. Then the combinatorial girth of $v$ in $X^{\star}$ is $7$.

\begin{figure}[h]
            \begin{center}
               \includegraphics[height=4cm]{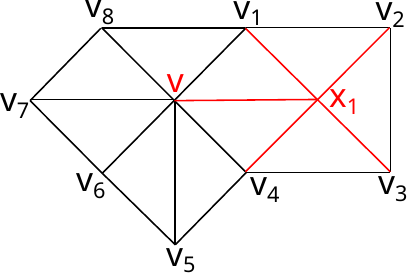}
               % \label{$w_{1} \sim w_{3}$}
              \caption{}
            \end{center}
        \end{figure}

\end{enumerate}

In conclusion the combinatorial girth in $X^{\star}$ of each vertex of a pentagon of $X$, is at least $7$. 
%Since, by Lemma \ref{3.1} $X^{\star}$ is flag, it is %$7$-located. 

%In order to show that $X^{\star}$ is locally $5$-%large, suppose by contradiction there exists a vertex %$v$ such that $|X^{\star}_{v}| \leq 4$. If %$|X^{\star}_{v}| = 3$, then the sum of vertices around %$v$ is $3 \cdot \cfrac{\pi}{3} = \pi < 2 \pi$. This %yields contradiction with $X^{\star}$ being a CAT(0) %space. If $|X^{\star}_{v}| = 4$, then the sum of %vertices around $v$ is $4 \cdot \cfrac{\pi}{3} < 2 %\pi$. This yields contradiction with $X^{\star}$ being %a CAT(0) space. Hence the combinatorial girth of the %link of each vertex is at least $5$. Therefore %$X^{\star}$ is locally $5$-large.

\end{proof}

%\begin{corollary}
%CB-graphs are $7$-located.
%\end{corollary}

\end{document}